\documentclass[11pt]{amsart}
\usepackage{geometry}                
\usepackage{graphicx}
\usepackage{amssymb}
\usepackage{epstopdf}
\usepackage{algorithm}
\usepackage{algorithmic}
\usepackage{dsfont}

\usepackage{amsthm}
\usepackage{tikz}
\usetikzlibrary{arrows,arrows.meta,backgrounds,calc,fit,decorations.pathreplacing,decorations.markings,shapes.geometric}
\tikzset{>=latex}

\newcommand{\badpairs}{{\sf BadPairs}}
\newcommand{\badpoints}{{\sf BadPoints}}
\newcommand{\PRS}{\text{PRS for Hard Disks}}

\newcommand{\law}{\mathcal D}
\newcommand{\lr}{\lambda_r}
\newcommand{\lrd}{\lambda_{r,d}}
\newcommand{\lc}{\bar\lambda}

\newcommand{\Sbar}{\overline S}

\let\rho=\varrho

\newtheorem{theorem}{Theorem}

\newtheorem{lemma}[theorem]{Lemma}

\newcommand{\Ex}{\mathop{\mathbb{{}E}}\nolimits}
\newcommand{\omegatilde}{\widetilde{\omega}}
\let\epsilon=\varepsilon
\newcommand{\calE}{\mathcal{E}}
\newcommand{\calF}{\mathcal{F}}
\newcommand{\origin}{\mathbf0}
\newcommand{\rset}{\mathbb{R}}

\begin{document}

\title[Perfect Simulation of the Hard Disks Model by Partial Rejection Sampling]{Perfect 
simulation of the Hard Disks Model\\ by Partial Rejection Sampling}
\thanks{The work described here was supported by the EPSRC research grant
EP/N004221/1 ``Algorithms that Count''.}




\author{Heng Guo}
\address[Heng Guo]{School of Informatics, University of Edinburgh, Informatics Forum, Edinburgh, EH8 9AB, United Kingdom.}
\email{hguo@inf.ed.ac.uk}

\author{Mark Jerrum}
\address[Mark Jerrum]{School of Mathematical Sciences, Queen Mary, University of London, Mile End Road, London, E1 4NS, United Kingdom.}
\email{m.jerrum@qmul.ac.uk}

\begin{abstract}
We present a perfect simulation of the hard disks model via the partial rejection sampling method.  
Provided the density of disks is not too high, the method produces exact samples in $O(\log n)$ rounds, and total time $O(n)$,
where $n$ is the expected number of disks.  The method extends easily to the hard spheres 
model in $d>2$ dimensions.
In order to apply the partial rejection method to this continuous setting,
we provide an alternative perspective of its correctness and run-time analysis that is valid for general state spaces.
\end{abstract}



\maketitle

\section{Introduction}

The \emph{hard disks model} is one of the simplest gas models in statistical physics.
Its configurations are non-overlapping disks of uniform radius~$r$ in a bounded region of~$\rset^2$.
For convenience, in this paper, we take this region to be the unit square $[0,1]^2$.
This model was precisely the one studied by Metropolis et al.~\cite{MRRTT53}, in their pioneering work on the Markov chain Monte Carlo (MCMC) method.
They used Los Alamos' MANIAC computer to simulate a system with $224$ disks.

There are two variants of this model.
To obtain the \emph{canonical ensemble}, we fix the number (or equivalently, density) of disks and decree that 
all configurations are ``equally likely'', subject only to the disks not overlapping.
In the \emph{grand canonical ensemble}, we fix the ``average'' number of disks.
To be more specific, centers of the disks are distributed according to a Poisson point process of intensity $\lr=\lambda/(\pi r^2)$,
conditioned on the disks being non-overlapping.
The hard disks model, and its higher dimensional generalization (called the \emph{hard spheres model}) 
are also related to the optimal sphere packing density~\cite{Hal05,Via17,CKMRV17}.
See \cite{JJP19,Coh17} and references therein for more details.
See also \cite{Low00} for the physics perspective.

Our main aim in this work is to describe and analyse a very simple algorithm 
for exactly sampling from the grand canonical ensemble, based on the partial rejection sampling
paradigm introduced by Guo, Jerrum and Liu~\cite{GJL17}.

More precisely, the challenge is the following:  produce a realisation~$P\subset[0,1]^2$
of a Poisson point process of intensity $\lr$ in the unit square, 
conditioned on the event that no pair of points in~$P$ are closer than $2r$ in Euclidean distance.  
We refer to this target measure as the \emph{hard disks distribution}. 
It describes an arrangement of open disks of radius~$r$ with centres in $[0,1]^2$ that are not
allowed to overlap, but which otherwise do not interact.  It is a special case of
the Strauss process~\cite{Str75}.  Note that, although the disks do not overlap each other, 
they may extend beyond the boundary of the unit square.  
Also, the intensity of the underlying Poisson process 
is normalised so that the expected number of points of~$P$ lying in a disk of radius~$r$ is~$\lambda$. 
This normalisation gives us sensible asymptotics as the radius of the disks tends to zero
(equivalently, the number of disks tends to  infinity).

Classical rejection sampling applied to this problem yields the  following 
algorithm:  repeatedly sample a realisation~$P$ of the Poisson process 
of intensity~$\lambda$ in the unit square until $P$ satisfies the condition that 
no two points are closer than~$2r$, and return~$P$.  Unfortunately, for every $\lambda>0$, 
however small,
the expected number of unsuccessful trials using this approach increases exponentially 
in $r^{-1}$, as $r\to0$.  Partial rejection sampling~\cite{GJL17} requires only a subset 
of $P$ to be resampled at each iteration.  Algorithm~\ref{alg:PRS} below arises 
from a routine application of the paradigm to the problem at hand.

The original partial rejection method \cite{GJL17} and its analysis are tailored for the discrete case.
In this paper we provide an alternative view on the correctness of the method, which is also valid in the continuous setting.
In other words, as with classical rejection sampling, Algorithm~\ref{alg:PRS} 
terminates with probability~1, producing a realisation of the exact hard disks distribution.

\begin{theorem}\label{thm:correctness}
Algorithm \ref{alg:PRS} is correct:  conditional on halting,
Algorithm \ref{alg:PRS} produces a sample from the hard disks distribution with
intensity $\lr=\lambda/(\pi r^2)$.
\end{theorem}

The proof of this result forms the content of Section~\ref{sec:correctness}.

\begin{figure}[h]
  \centering
  \includegraphics[width=200pt]{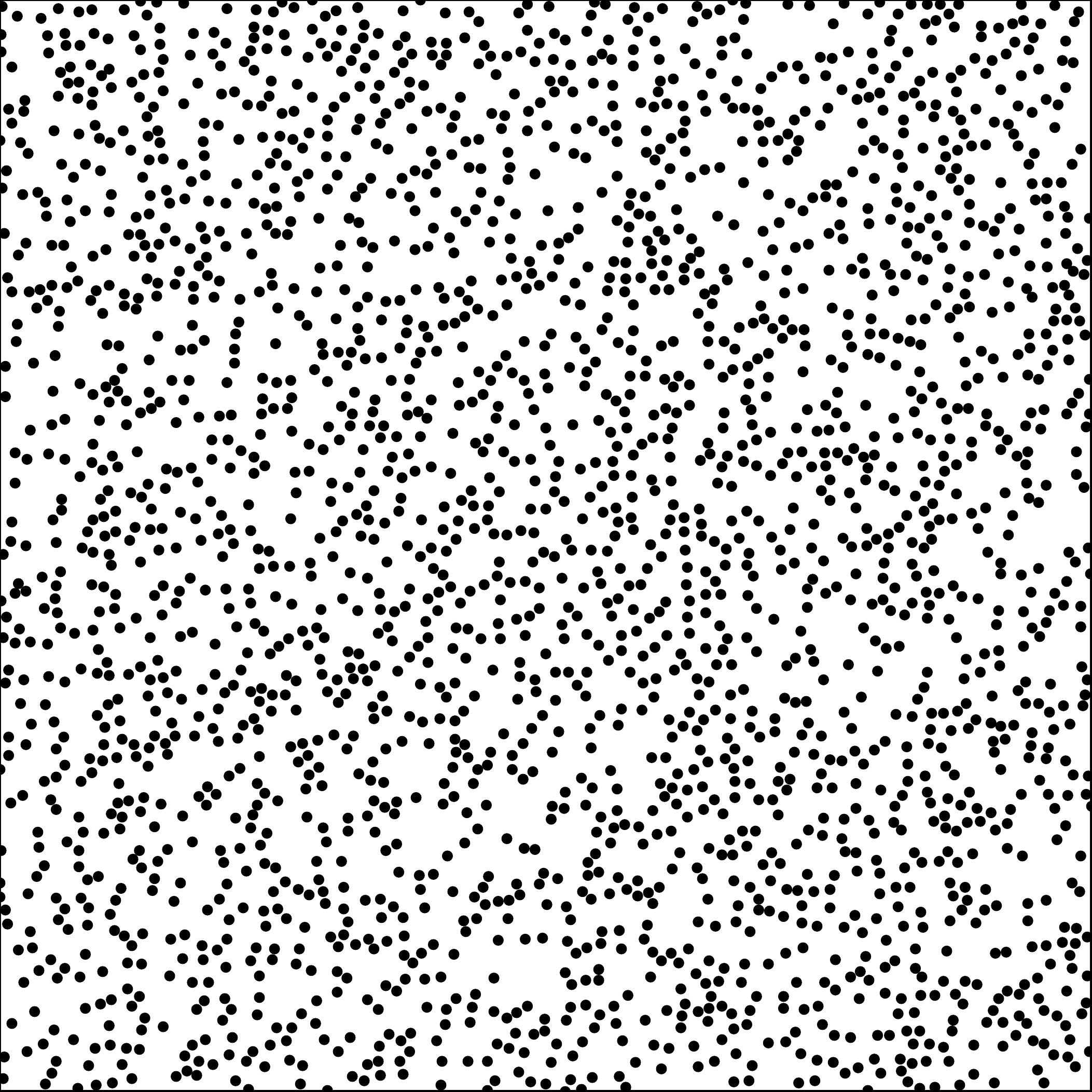}
  \caption{A realisation of the hard disks measure with $\lambda=0.5$ and $r=\frac{1}{200}$. The resulting density is $\alpha=0.189+$.}
  \label{fig:harddisks}
\end{figure}

In contrast to classical rejection sampling, the expected number 
of iterations (resampling steps) is now asymptotically $O(\log(r^{-1}))$ as $r\to0$,
provided $\lambda$ is not too large. 
Furthermore, with a suitable implementation, the total runtime is $O(r^{-2})$, i.e., linear in the number of disks.
 We prove that rapid termination occurs when $\lambda<0.21027$. 
This analysis is not tight, and experiments suggest that the actual
threshold for rapid termination is around~$\lambda\approx0.5$.
Figure~\ref{fig:harddisks} is a realisation of $\lambda =0.5$ with $r=\frac{1}{200}$.
The resulting density is $\alpha=0.189+$.

\begin{theorem}\label{thm:runtime}
Fix $\lambda<0.21027$.  Then the expected number of iterations of the while-loop
in Algorithm~\ref{alg:PRS} is $O(\log r^{-1})$.  Moreover, with a suitable implementation, 
the overall runtime of the algorithm is $O(r^{-2})$.
\end{theorem}

The proof of this result forms the content of Section~\ref{sec:runtime}.

The method extends naturally to the \emph{hard spheres model} in $d>2$ dimensions. 
Here, the desired distribution is a Poisson point process in $[0,1]^d$ conditioned
on no pair of points being closer than $2r$.   
The natural normalisation for the intensity of the Poisson process
in $d$ dimensions is $\lrd=\lambda/(v_dr^d)$,
where $v_d$ is the volume of a ball of radius~$1$ in $\rset^d$.   
With this convention, we
prove that rapid termination occurs in $d$~dimensions provided $\lambda<2^{-(d+\frac12)}$.

The \emph{expected packing density} $\alpha(\lambda)$ or simply $\alpha$ for this model 
is the expected total volume of spheres.  (Note that, neglecting boundary effects, $\alpha$ is 
the proportion of the unit cube occupied by spheres.)  
The quantity $\alpha(\lambda)$ grows monotonically with~$\lambda$, 
but intuitively we expect its rate of growth to slow down dramatically as the spheres pack more tightly.
The connection between expected packing density~$\alpha$ 
and intensity~$\lambda$ has recently been thoroughly 
explored by Jenssen, Joos and Perkins~\cite{JJP19}.  Using their results, we show that 
partial rejection sampling can achieve expected packing density $\Omega(2^{-d})$ while retaining
rapid termination in $O(\log(r^{-1}))$ iterations.  Although sphere packings of density $\Omega(d2^{-d})$
have been proved to exist, there is no polynomial-time sampler that provably
achieves packing density beyond $O(2^{-d})$, as far as we are aware. 

Other approaches to exact sampling include Coupling From The Past (CFTP), which
was adapted to point processes by Kendall~\cite{Kendall98} and Kendall and M{\o}ller \cite{KM00}.
Recently, Moka, Juneja and Mandjes~\cite{MJM17} proposed an algorithm based on rejection and importance sampling. 
Although their algorithm, like ours, is based on rejection sampling,    
it does not share our asymptotic performance guarantee.   Indeed, 
its runtime appears to grow exponentially as the number of disks goes to infinity, 
with the density of disks held constant.

The most widely used approach to sampling configurations from the hard disks model 
is Markov chain Monte Carlo (MCMC){}.  
Here the desired distribution is approached in the limit as the runtime~$t$ of the sampler tends to infinity.  
In this sense, MCMC produces an approximate sampler, though the error (in total variation distance) decays 
exponentially with~$t$.  The problem lies in deciding how large $t$ should be in order to ensure that the samples 
obtained are close enough to the desired distribution.  There are two possibilities.  
The runtime~$t$ may be chosen heuristically, in which
case the quality of the output from the sampler is not guaranteed.  
Otherwise, an analytical upper bound on mixing time may be used to determine a suitable~$t$, 
but then the tendency is for this bound to be very conservative.    
The experimental advantage of partial rejection sampling is its simple termination rule,
combined with the property of generating perfect samples from the desired distribution.    

Approximate sampling via Markov chain simulation has been studied by Kannan, Mahoney 
and Montenegro~\cite{KMM03} and Hayes and Moore~\cite{HM14} in the context 
of the canonical ensemble, where the number of spheres in a configuration is fixed.
(So in the MCMC approach, it is the density $\alpha$ that is chosen in advance, while in 
the approach taken here we choose $\lambda$ in advance and then $\alpha$ follows.)   
Kannan et al.~\cite{KMM03} show that rapid mixing (convergence to near-stationarity in 
time polynomial in the number of disks) occurs for densities $\alpha\leq2^{-(d+1)}$, 
in dimension~$d$.  The best rigorously derived density bound guaranteeing rapid mixing 
in two dimensions is given by Hayes and Moore~\cite{HM14} and is $\alpha\approx 0.154$.
(Note that this improves on the $\alpha=\frac18$ obtained by Kannan et al.)
It is should be noted that these results are not directly comparable with our~$\lambda<0.21027$ 
owing to the difference in models.
To obtain canonical ensembles, we could use Algorithm~\ref{alg:PRS} and further condition 
on the number of desired disks.
However, the only rigorous guarantee for this approach, via \cite{JJP19}, is $\alpha(0.21027)> 0.0887$.

It is believed that the hard disks model in two dimensions undergoes a phase transition at a certain
critical density $\alpha_c$:  at lower densities configurations are disordered, while at higher densities, 
long range correlations can be observed.  Unfortunately, $\alpha_c$ is well beyond the densities that 
can be achieved by samplers (either based on MCMC or rejection sampling) with rigorous performance guarantees.
However, heuristic approaches using sophisticated MCMC samplers~\cite{EngelEtAl} suggest that the critical density is around
$\alpha_c\approx 0.7$.

Finally, determining the maximum achievable density $\alpha$ is the classical sphere packing problem.
Despite extensive study, the exact solution is known only for dimensions $d=1,2,3,8$ and 24.
See \cite{JJP19, Per16} for rigorous bounds on packing densities in general dimension~$d$.

A preliminary version of this paper was presented at ICALP 2018~\protect\cite{HardDisksICALP}.
Subsequently, the constant in Theorem~\ref{thm:runtime} has been improved by Jake Wellens \cite{Wel18} to $0.2344+$.

\section{The sampling algorithm}

The following notation will be used throughout.
If $P$ is a finite subset of $[0,1]^2$ then
\[\badpairs(P)=\big\{\{x,y\}:x,y\in P\wedge x\not=y\wedge\|x-y\|<2r\big\},\]
where $\|\cdot\|$ denotes Euclidean norm, and
\[\badpoints(P)=\bigcup\badpairs(P).\]
The open disk of radius $r$ with centre $x\in[0,1]^2$ is denoted by $D_{r}(x)$.
The finite set $\Pi\subset[0,1]^2$ always denotes a realisation of the Poisson point
process of intensity $\lr$ on $[0,1]^2$.
For a random variable~$X$ and event~$\calE$ we use $\law(X)$ to
denote the distribution (law) of~$X$, and $\law(X\mid \calE)$ the
distribution of $X$ conditioned on~$\calE$ occurring.
Thus, $\law(\Pi\mid\badpoints(\Pi)=\emptyset)$ is the hard disks distribution
that we are interested in.

Our goal is to analyse the correctness and runtime of a sampling algorithm 
for the hard disks model
(see Algorithm~\ref{alg:PRS} below),
specifically to determine the largest value of $\lambda$ for which it terminates
quickly, i.e., in $O(\log r^{-1})$ iterations.
The algorithm is an example application of ``Partial Rejection Sampling''~\cite{GJL17}, adapted
to the continuous state space setting.

\begin{algorithm}
\caption{Partial Rejection Sampling for the hard disks model}\label{alg:PRS}
\begin{algorithmic}
\STATE $\PRS(\lambda,r)$
\COMMENT {$r$ is the disk radius and $\lr=\lambda/(\pi r^2)$ the intensity}
\STATE {Let $P$ be a sample from the Poisson point process
        of intensity $\lr$ on the unit square}
\WHILE {$B\gets\badpoints(P)\neq\emptyset$}
\STATE {$S\gets B+D_{2r}(\origin)$}
\COMMENT {Resampling set is the Minkowski sum of $B$ with a disk of radius $2r$}
\STATE {Let $P^S$ be a sample from the Poisson point process
        of intensity $\lr$ on~$S$}
\STATE {$P\gets(P\setminus B)\cup P^S$}
\ENDWHILE
\RETURN $P$
\end{algorithmic}
\end{algorithm}

\section{Correctness (Proof of Theorem~\ref{thm:correctness})}\label{sec:correctness}

Let $B$ be any finite subset of $[0,1]^2$. We say that $B$ is
a \emph{feasible set} of bad points if $\badpoints(B)=B$;
this is equivalent to saying that there is
a finite subset $R\subset[0,1]^2$ with $B=\badpoints(R)$.
The key to establishing correctness of Algorithm~\ref{alg:PRS} is the following loop invariant:
\[\law(P\mid \badpoints(P)=B)=\law(\Pi\mid\badpoints(\Pi)=B),\]
for every feasible set~$B$,
where $P$ is any intermediate set of points during the execution of the algorithm,
and $\Pi$ is  a realisation of the Poisson point process.
Let us consider what the right hand side means operationally.
Let $S=B+D_{2r}(\origin)$.
(This is the resampling set used by the algorithm.)  Let $Q$ be a sample from the
distribution $\law(\Pi\mid\badpoints(\Pi)=B)$.  The only points in~$Q$ that
lie inside~$S$ are the points in~$B$.  (Any extra points would create more bad
pairs than there actually are.)  Thus $Q\cap S=B$.
Outside of $S$ there are no bad pairs;  thus $Q\cap\Sbar$ is a sample from the
hard disks distribution on~$\Sbar=[0,1]^2\setminus S$.
Note that, setting $B=\emptyset$, we see that $\law(\Pi\mid\badpoints(\Pi)=\emptyset)$
is just the hard disks distribution on $[0,1]^2$.

Let $T$ (a random variable) be the number of iterations of the while-loop.
On each iteration, the while loop terminates with probability bounded away from~0;
thus $T$ is finite with probability~1.  (Indeed, $T$ has finite expectation.)
Let $P_t$, for $1\leq t\leq T$,
be the point set~$P$ after $t$~iterations of the loop, and $P_0$
be the initial value of $P$ (which is just a realisation of the Poisson point process
on $[0,1]^2$).
We say that $B_0,B_1,\ldots,B_t\subset[0,1]^2$ is a \emph{feasible sequence} of (finite)
point sets if there exists a run of Algorithm~\ref{alg:PRS} with
$\badpoints(P_0)=B_0,\ldots,\badpoints(P_t)=B_t$.
Theorem~\ref{thm:correctness} will follow easily from the following lemma.

\begin{lemma}\label{lem:main}
Let $B_0,B_1,\ldots,B_t\subset[0,1]^2$ be a feasible sequence.  Then
\begin{align*}
&\law\big(P_t\bigm|\badpoints(P_0)=B_0\wedge\cdots\wedge\badpoints(P_t)=B_t\big)\\
&\qquad=\law(P_t\bigm|\badpoints(P_t)=B_t)\\
&\qquad=\law(\Pi\bigm|\badpoints(\Pi)=B_t).
\end{align*}
\end{lemma}

\begin{proof}
We prove the result by induction on~$t$.
The base case, $t=0$, holds by construction: $P_0$ is just a realisation of the Poisson
point process on $[0,1]^2$.
Our induction hypothesis is
\begin{equation}\label{eq:IH}
\law\big(P_t\bigm|\badpoints(P_0)=B_0\wedge\cdots\wedge\badpoints(P_t)=B_t\big)
=\law(\Pi\bigm|\badpoints(\Pi)=B_t),
\end{equation}
for every feasible sequence $B_0,\ldots,B_t$.
Extend the feasible sequence to $B_{t+1}$.
For the inductive step, we assume (\ref{eq:IH}) and aim to derive
\begin{multline}\label{eq:IS}
\law\big(P_{t+1}\bigm| \badpoints(P_0)=B_0\wedge\cdots\wedge\badpoints(P_{t+1})=B_{t+1}\big)\\
=\law(\Pi\mid\badpoints(\Pi)=B_{t+1}).
\end{multline}
The resampling set on iteration $t+1$ is $S=B_t+D_{2r}(\origin)$.
As a first step we argue below that
\begin{multline}\label{eq:step1}
\law\big(P_{t+1}\bigm| \badpoints(P_0)=B_0\wedge\cdots\wedge\badpoints(P_t)=B_t\big)\\
=\law(\Pi\mid\badpairs(\Pi)\cap\Sbar^{(2)}=\emptyset),
\end{multline}
where $\Sbar^{(2)}$ denotes the set of unordered pairs of elements from~$\Sbar$.
We have noted that (\ref{eq:IH}) implies that, outside of the resampling set~$S$,
the point set~$P_t$ is a realisation of the hard disks distribution.
Also, the algorithm does not resample points outside of~$S$.
Thus $P_{t+1}\cap\Sbar=P_t\cap\Sbar$ is Poisson distributed, conditioned on there
being no bad pairs.  Inside $S$, resampling has left behind
a fresh Poisson point process $P_{t+1}\cap S$.
These considerations give (\ref{eq:step1}).

Next, we condition on $B_{t+1}$:
\begin{multline*}
\law\big(P_{t+1}\bigm|\badpoints(P_0)=B_0\wedge\cdots
\wedge\badpoints(P_t)=B_t\wedge\badpoints(P_{t+1})=B_{t+1}\big)\\
=\law\big(\Pi\bigm|\badpairs(\Pi)\cap\Sbar^{(2)}=\emptyset\wedge\badpoints(\Pi)=B_{t+1}\big).
\end{multline*}
Since $B_{t+1}$ contains no bad pairs with both endpoints in $\Sbar$, 
the event $\badpoints(\Pi)=B_{t+1}$ entails the event
$\badpairs(\Pi)\cap\Sbar^{(2)}=\emptyset$.  Thus, we have
\begin{multline*}
\law\big(P_{t+1}\bigm|\badpoints(P_0)=B_0\wedge\cdots
\wedge\badpoints(P_t)=B_t\wedge\badpoints(P_{t+1})=B_{t+1}\big)\\
  =\law(\Pi\mid\badpoints(P_{t+1})=B_{t+1}).
\end{multline*}
The right hand side of this equation does not involve $B_0,\ldots,B_t$, and so
\[\law(P_{t+1}\mid\badpoints(P_{t+1})=B_{t+1})=\law(\Pi\mid\badpoints(\Pi)=B_{t+1}).\]
This completes the induction step~(\ref{eq:IS}) and the proof.
\end{proof}

We can now complete the proof of Theorem~\ref{thm:correctness}.
As we observed earlier, $T$, the number of iterations of the while-loop,
is finite with probability~1.  
By Lemma~\ref{lem:main}, noting that $B_T=\emptyset$,
\[
\law(P_T)=\law(\Pi\mid\badpoints(\Pi)=\emptyset).
\]
In other words, at termination, Algorithm~\ref{alg:PRS} produces a realisation
of the hard disks distribution on $[0,1]^2$.

\section{Run-time analysis (Proof of Theorem~\ref{thm:runtime})}\label{sec:runtime}

We consider how the number of ``bad events'',
i.e., the cardinality of the set $\badpairs(P_t)$,
evolves with time.
As usual $\Pi$ denotes a realisation of the Poisson point process
of intensity~$\lr$.  Also denote by $\Delta$ a realisation of the hard disks
process of the same intensity.  We need the following stochastic domination result.

\begin{lemma}\label{lem:dominate}
The hard disks distribution is stochastically dominated by the Poisson point
process with the same intensity.  That is, we can construct a joint sample
space for $\Pi$ and $\Delta$ such that $\Delta\subseteq\Pi$.
\end{lemma}

Holley's criterion is a useful test for stochastic domination, but it
is not of direct use to us in the proof of Lemma~\ref{lem:dominate}, 
because it applies only to finite state spaces.
Fortunately, Preston~\cite[Theorem 9.1]{Preston}, has derived
a version of Holley's criterion that fits our situation.  We will mostly follow
Preston's notation, except that, to save confusion, we will use
$P$ and $Q$, rather than $x$ and~$y$, to denote finite sets of points.
In order to state his result, we need some notation.  In our application,
$\omegatilde_n$ is the distribution
on $([0,1]^2)^{(n)}$ obtained by sampling $n$~points independently
and uniformly at random from $[0,1]^2$, and regarding the points
as indistinguishable;
furthermore, $\omegatilde=\sum_{n=0}^{\infty}\omegatilde_n/n!$.  (For consistency with Preston,
we have left $\omegatilde$ unnormalised.  If we had made it into a probability distribution by
division by~$e$, then $\omegatilde$ could be thought of as follows:  
sample an integer $k$ from the Poisson distribution with mean~1, and then pick $k$ (unlabelled) points 
uniformly and independently.) 
Denote by
$\Omega$ the set of all finite subsets of $[0,1]^2$, and 
by $\calF$ the set of non-negative measurable functions 
$\Omega\to\mathbb{R}$ satisfying
\begin{equation}\label{eq:normalise}
\int f\,d\omegatilde=1,
\end{equation}
and 
\begin{equation}\label{eq:ideal}
\text{$f(P)=0$ and $P\subseteq Q$ implies $f(Q)=0$.}
\end{equation}
(See Preston~\cite[Section 9]{Preston} for detailed formal definitions of the concepts here.)

\begin{lemma}[Theorem 9.1 of \cite{Preston}]\label{lem:Preston}
Let $f_1,f_2\in\calF$ and suppose that 
\begin{equation}\label{eq:premise}
\frac{f_1(P+\xi)}{f_1(P)}\geq \frac{f_2(Q+\xi)}{f_2(Q)}, 
   \quad\text{for all $Q\subseteq P\in\Omega$ and $\xi\in[0,1]^2\setminus P$}
\end{equation}
(where, by convention, $0/0=0$).  Then, for all bounded, measurable, non-decreasing functions 
$g:\Omega\to\mathbb{R}$,
\[
\int gf_1\,d\omegatilde\geq\int gf_2\,d\omegatilde,
\]
i.e., if $\mu_i$ is the probability measure having density $f_i$ with respect to $\omegatilde$,
then $\mu_1$ stochastically dominates $\mu_2$.
\end{lemma}
\begin{proof}[Proof of Lemma \ref{lem:dominate}]
We set
\begin{align*}
f_1(P)&=C_1\lr^{|P|}\\
\noalign{\noindent and}
f_2(P)&=\begin{cases}
C_2\lr^{|P|},&\text{if $\badpairs(P)=\emptyset$;}\\
0,&\text{otherwise.}
\end{cases}
\end{align*}
The normalising constants $C_1$ and $C_2$ are
chosen so that both $f=f_1$ and $f=f_2$
satisfy \eqref{eq:normalise}.
(There is an explicit expression for~$C_1$, namely $C_1=\exp(-\lr)$,
but not for~$C_2$.)
Notice that both $f_1$ and $f_2$ also satisfy \eqref{eq:ideal}.
Notice also that the probability measures $\mu_1$ and~$\mu_2$
of the Poisson point process
and the hard disks process have densities $f_1$ and $f_2$ with respect to~$\omegatilde$.
The premise \eqref{eq:premise} of Lemma~\ref{lem:Preston}
holds, since the left-hand side is always $\lr$ and the right-hand side
is either $\lr$ or~0.
The conclusion is that $\mu_1$ dominates $\mu_2$.
Strassen's Theorem~\cite{Lindvall,Strassen},
allows us to conclude the existence of a coupling of $\Pi$ and $\Delta$
as advertised in the statement of the lemma (except, possibly, on a set of measure zero).
\end{proof}

\begin{lemma}\label{lem:runtime}
There is a bound $\lc>0$ such that the expected number of iterations of the while-loop
in Algorithm~\ref{alg:PRS} is $O(\log r^{-1})$ when $\lambda<\lc$.
\end{lemma}

\begin{proof}
First observe that $\badpairs(P)$ determines $\badpoints(P)$ and vice versa.
So conditioning on the set $\badpairs(P)$ is equivalent to conditioning
on $\badpoints(P)$.

Introduce random variables $Z_t=|\badpairs(P_t)|$, for $0\leq t\leq T$.
Our strategy is to show that
\begin{equation}\label{eq:martingale}
\Ex (Z_{t+1}\mid Z_0,\ldots,Z_t)\leq \alpha^{-1} Z_t,
\end{equation}
for some $\alpha>1$.  Then $Z_0,\alpha Z_1,\alpha^2 Z_2,\alpha^3 Z_3,\ldots$ is
a supermartingale (with the convention that $Z_t=0$ for all $t>T$).
Therefore,
$\Ex Z_t\leq \alpha^{-t}\Ex Z_0\leq \frac12\lr^2 \alpha^{-t}$.
Here, we have used the fact that $|P_0|$ is a Poisson random variable with
expectation $\lr$, and 
\[Z_0=|\badpairs(P_0)|\leq\tfrac12|P_0|\,(|P_0|-1),\]
and hence 
\[\Ex Z_0\leq\Ex\big[|P_0|\,(|P_0|-1)\big]=\tfrac12\lambda_r^2.\]
Setting $t=O(\log r^{-1}+\log \epsilon^{-1})$, we obtain $\Ex Z_t<1/\epsilon$,
and hence $\Pr(T>t)\leq \epsilon$.  It follows that the
expected number of iterations of the while-loop of Algorithm~\ref{alg:PRS}
is $O(\log r^{-1})$.  Note that the probability of non-termination decreases
exponentially with~$t$, so the probability of large deviations above the expected
value of~$T$ is low.

Crude estimates give $\lc=1/(4\sqrt 2)$.  The calculation goes as follows.
Suppose, in \eqref{eq:martingale}, we condition on the random variables 
$\badpoints(P_0),\ldots,\allowbreak\badpoints(P_t)$, rather than simply on $Z_0,\ldots,Z_t$.  
This is more stringent conditioning, since the former random variables determine 
the latter.  It is enough to establish \eqref{eq:martingale} under the more stringent 
conditioning.  So fix $\badpoints(P_0)=B_0,\ldots,\badpoints(P_t)=B_t$, and note that
this choice also fixes the resampling sets $S_0,\ldots,S_t$.
Suppose $Z_t=|\badpairs(P_t)|=k$.  Inside the resampling set $S_t$ we
have a Poisson point process $P_{t+1}\cap S_t$ of intensity $\lr$.
Outside, by Lemma~\ref{lem:main}, 
there is a realisation of the hard disks process.  Since we are seeking
an upper bound on $Z_{t+1}$ we may, by Lemma~\ref{lem:dominate},
replace $P_{t+1}\cap\Sbar_t$ by a Poisson point process of intensity~$\lr$.

Let $k'=\Ex(Z_{t+1}\mid Z_0,\ldots,Z_t)$.  From the above considerations we have
\begin{equation}\label{eq:doubleint}
  k'\leq \int_{S_t}\lr\int_{[0,1]^2}\lr\,\mathbf{1}_{\|x-y\|\leq2r}\,dy\,dx.
\end{equation}
This is an overestimate, as we are double-counting overlapping disks
whose centres both lie within~$S_t$.   Now, $S_t$ is a union of at most $2k$ disks of
radius~$2r$.  Thus
\begin{align}
k'&\leq 2k\lr^2\int_{D_{2r}(\origin)}\int_{\rset^2}\mathbf{1}_{\|x-y\|\leq2r}\,dy\,dx\label{eq:kprimeineq}\\
&=2k\lr^2\int_{D_{2r}(\origin)}\int_{D_{2r}(x)}dy\,dx\notag\\
&=2k\lr^2\times 4\pi r^2 \times 4\pi r^2\notag\\
&=32\lambda^2k.\notag
\end{align}
There are further sources of slack here:  there may be fewer than $2k$ disks,
the disks comprising $S_t$ certainly
overlap, and, for points $x$ near the boundary, some of disks $D_{2r}(x)$ will lie
partly outside the unit square. (The last of these presumably has no effect asymptotically,
as $r\to0$.)
Setting $\lc=1/(4\sqrt 2)=0.17677+$, we see that $\alpha=k/k'>1$ for any $\lambda<\lc$,
and $(Z_t)_{t=0}^\infty$ is a supermartingale, as required.
\end{proof}

The constant $\lc$ may seem quite small.  Note, however, that classical rejection sampling
cannot achieve any $\lc>0$.  The argument goes as follows.  Divide $[0,1]^2$
into $r\times r$ squares.  If there are two points in the same square then they will certainly
be less than distance $2r$ apart.  The number of points in each square is Poisson distributed
with mean $\lambda/\pi$.  Thus for any $\lambda>0$ the probability that a particular
square has at least two points is bounded away from zero.  The number of points in
each square is independent of all the others.
It follows that the runtime of classical rejection sampling is exponential in $r^{-2}$.

The above derivation for $\lc$ is quite crude and can be improved.
\begin{lemma}\label{lem:runtime-2}
The constant $\lc$ in Lemma~\ref{lem:runtime} can be taken to be $0.21027$.
\end{lemma}
\begin{proof}
For each of the $2k$ disks, the right-hand side of inequality \eqref{eq:kprimeineq} 
counts pairs of points $(x,y)$ with $x$ in the disk, and $y$ anywhere within
distance $2r$ of~$x$.  Since a bad event is determined by an \emph{unordered}  
pair of points, this gives rise to significant double counting.  In particular,
pairs $(x,y)$ with $x$ and $y$ lying in the same ball are double counted.  We can 
subtract off these pairs to obtain a better estimate.

For a single ball, the correction is 
\[
C=\frac12\int_{D_{2r}(\origin)}\lr\int_{D_{2r}(\origin)}\lr\mathbf{1}_{\|x-y\|\leq2r}\,dy\,dx.
\]
(The initial factor of one half arises because we want to count unordered pairs.)
With the change of variables $x=2rx'$ and $y=2ry'$ this expression simplifies to 
\begin{align*}
C&=\frac12\times 16r^4\lr^2\int_{D_{1}(\origin)}\int_{D_{1}(\origin)}\mathbf{1}_{\|x'-y'\|\leq1}\,dy'\,dx'\\
&=8\lr^2r^4\int_{D_{1}(\origin)}L(\|x'\|)\,dx',
\end{align*}
where $L(\|x'\|)$ is the area of the ``lens'' $D_{1}(\origin)\cap D_{1}(x')$.
Letting $\rho$ denote the offset of the centres of the two disks, the area of the lens is given by 
\[
L(\rho)=2\arccos(\rho/2)-\tfrac12\rho\sqrt{4-\rho^2}.
\]
(This is by elementary geometry:  the lens is the intersection of 
two sectors, one from each of the disks, and its area can be 
computed by inclusion-exclusion.)
An illustration (before shifting $y$ to $\mathbf{0}$) is given in Figure~\ref{fig:illustration-offset}.
The shaded area is $L$.

\begin{figure}[htbp]
  \centering
    \begin{tikzpicture}[scale=0.7, inner sep=2pt, transform shape]
      \draw (0,0) ellipse (3 and 3) [thick];
      \draw (0,0) node [draw,fill,shape=circle,color=black, label=0:{\LARGE $x$}] (x) {}; 
      \draw (-2,0) ellipse (3 and 3) [thick];
      \draw (-2,0) node [draw,fill,shape=circle,color=black, label=180:{\LARGE $y$}] (y) {}; 
      \draw (-1,-1.5) node {\LARGE $L$};

      \draw (x) edge [<->] node [label=90:{\LARGE $\rho$}] {} (y);
      \draw (x) edge [->] node [label=-15:{\LARGE $2r$}] {} (2.121,2.121);
       \begin{scope}[on background layer]
        \clip (0,0) ellipse (3 and 3);
        \clip (-2,0) ellipse (3 and 3);
        \fill[black!20!white] (0,0) circle [radius = 3];
      \end{scope}
    \end{tikzpicture}  
  \caption{An illustration of double counting.}
  \label{fig:illustration-offset}
\end{figure}
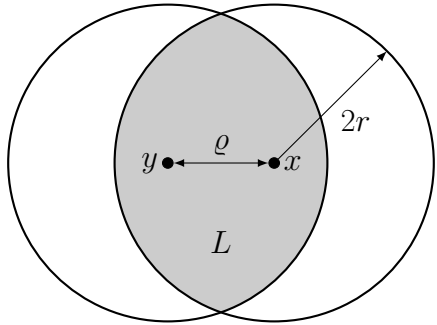

Translating to polar coordinates $(\rho,\theta)$, 
\begin{align*}
C&=8\lr^2r^4\int_{D_1(\origin)}L(\|x'\|_2)\,dx'\\
&=8\lr^2r^4\int_0^{2\pi}\!\!\int_0^1\rho\,L(\rho)\,d\rho\,d\theta\\
&=\frac{8\lambda^2}{\pi^2}\times2\pi\int_0^1\rho\,L(\rho)\,d\rho\\
&=\frac{16\lambda^2}{\pi}
  \left[\frac\pi2+(\rho^2-1)\arccos\Big(\frac\rho2\Big)
  -\Big(\frac\rho4+\frac{\rho^3}8\Big)\sqrt{4-\rho^2}\,\right]_0^1\\
&=\bigg(8-\frac{6\sqrt3}\pi\,\bigg)\lambda^2.
\end{align*}
(The integral was evaluated using the Maple computer algebra system.)  
Our revised upper bound on $k'$ is thus
\begin{align}\label{eqn:contraction}
k'\leq 2k(16\lambda^2-C)=2k\lambda^2\bigg(8+\frac{6\sqrt3}\pi\,\bigg),
\end{align}
Solving
\[
\lc^2\bigg(16+\frac{12\sqrt3}\pi\bigg)=1
\]
yields the improved bound of $\lc=0.21027+$.
\end{proof}
There are other factors that could in principle be used to increase $\lc$ further ---
each disk necessarily overlaps with at least one other disk, some bad events are triple
or quadruple counted --- but the computational difficulties rapidly increase
when attempting to account for these.

We have just seen that the number of iterations of the while-loop in Algorithm~\ref{alg:PRS} is small 
when $\lambda$ is below some threshold.  It just remains to check that the body of the loop can be implemented 
efficiently.  

\newcommand\bfi{\mathbf{i}}
\newcommand\bfj{\mathbf{j}}
\newcommand\infnorm[1]{\|#1\|_\infty}
\newcommand\pairs{\mathcal{P}}
\newcommand\calI{\mathcal{I}}

\begin{lemma}\label{lem:runtimeloop}
There is an implementation of Algorithm~\ref{alg:PRS} such that, for any fixed $\lambda<\lc$,  
the expected runtime of the algorithm is $O(r^{-2})$.
\end{lemma}
\begin{proof}
No sophisticated data structures are required, but the runtime analysis requires some work.

Divide the unit square into a grid of $r\times r$ squares.  
Index the grid squares by $\calI=[0,n]^2$, where $n\approx r^{-1}$.  
Let the grid squares be $\{\Gamma_\bfi:\bfi\in \calI\}$.  
Note that if two points of a point set $P$ lie in the same grid square 
then they will necessarily form a bad pair and lie in $\badpairs(P)$.  
Moreover, if $(x,y)\in\badpairs(P)$ and $x\in \Gamma_\bfi$ and $y\in\Gamma_\bfj$ then necessarily $\infnorm{\bfi-\bfj}\leq2$.  
So it seems, intuitively, that we should be able to implement the body of loop to run in time $O(n^2)=O(r^{-2})$, 
since we only have to examine $O(n^2)$ pairs of grid squares in order to complete the computation.  
Given Lemma~\ref{lem:runtime}, this would give an upper bound on total runtime of $O(r^{-2}\log r^{-1})$.

However, we can do better than this by noting that the amount of work to be done in each iteration decays exponentially. 
Recall the notation employed in Algorithm~\ref{alg:PRS}.
Assume that, in addition to $P$ and $B=\badpoints(P)$, we maintain a list~$L$ of grid squares containing points in~$B$.  
In order to perform the computations within the while-loop, it is only necessary to consider grid squares that are within constant distance of a grid square in~$L$.
As the length of the list~$L$ decays exponentially, according to Lemma~\ref{lem:runtime}, 
we expect the overall runtime to be $O(r^{-2})$ rather than just $O(r^{-2}\log r^{-1})$.  
This is indeed the case.  
However, there is a technical issue;  the time taken for the $t$th iteration of the loop is a random variable, say $T_t$.  
Recall the notation of Lemma~\ref{lem:runtime}, in particular that $Z_t=|\badpairs(P_t)|$ denotes the number of bad pairs
of points (ones within distance $2r$ of each other) after $t$ iterations of the loop.  
As may be expected, we can bound the expectation of~$T_t$ by a linear function of $Z_{t-1}$.  
Unfortunately, we can't just sum these estimates over $t\in\{1,T\}$ to get an upper bound on expected total runtime,
as the random variable $T_t$ is presumably correlated with $Z_t$.  
Instead, we shall ``charge'' the operations within the $t$th iteration of the loop either to $Z_{t-1}$ or $Z_t$.
and hence bound the expected runtime by $O(Z_0+Z_1+\cdots+Z_T)=O(r^{-2})$.

Consider the work done during the $t$th iteration of the while-loop.  
Assume that the point sets $P=P_{t-1}$ and~$B=B_{t-1}$ are available either from the initialisation phase 
or from the previous iteration of the loop.  
Also assume that we have the list~$L$ containing grid squares containing 
points in~$B$.  Note that $|L|\leq 2Z_{t-1}$, since each bad pair contributes at most two bad points.  
To compute $P^S$, we could sample a realisation of the Poisson point process in the unit square
and reject all those points that are not within distance $2r$ of some point in~$B$. 
However, this would be inefficient.   Instead, we just produce realisations of a Poisson point process
within grid squares that are within distance two of a grid square in~$L$. 
Here, we take the distance between grid squares $\Gamma_\bfi$ and $\Gamma_\bfj$ to be 
$\infnorm{\bfj-\bfi}$.
We then reject each new point unless it is within distance $2r$ of a point in~$B$;  the result is the 
required point set~$P^S$.

Let 
\begin{align*}
b_\bfi&=|B_{t-1}\cap\Gamma_\bfi|,\\
n_\bfi&=|P^S_t\cap\Gamma_\bfi|,\quad\text{and}\\
p_\bfi&=|P_t\cap\Gamma_\bfi|.
\end{align*}
Thus, restricted to grid square $\Gamma_\bfi$, we have that
$b_\bfi$ is the number of bad points carried forward from the previous iteration,
$n_\bfi$ is the number of fresh points added during the current iteration, 
and $p_\bfi$ is the total number of points at the end of the current iteration.
Also, let $n'_\bfi$ be the number of points generated during the current 
iteration that did not survive because they were not within distance $2r$ of some bad point.

Let $\pairs=\{(\bfi,\bfj):b_\bfi>0\text{ and }\infnorm{\bfj-\bfi}\leq2\}$.  
Note that, in order to compute the points that need to be added to~$P$ during the 
current iteration, only pairs of grid squares
$\{(\Gamma_\bfi,\Gamma_\bfj):(\bfi,\bfj)\in\pairs\}$ 
need to be examined.  (Each new point must be within distance $2r$ of an existing bad point.)
The work done during iteration $t$ in computing $P_t$ is then proportional to
$$
\sum_{(\bfi,\bfj)\in\pairs}1 + b_\bfi(n_\bfj+n'_\bfj)=A+B+C,
$$
where 
$$
A=\sum_{(\bfi,\bfj)\in\pairs}1,\quad B=\sum_{(\bfi,\bfj)\in\pairs}b_\bfi n_\bfj\quad\text{and}\quad
C=\sum_{(\bfi,\bfj)\in\pairs}b_\bfi n'_\bfj.
$$
Here, $A$ represents the fixed cost of cycling through all the relevant pairs of grid squares, and $A+B$
is the additional time required to examine each fresh point and see whether it needs to be retained 
(i.e., added to $P^S$).  

Now $|\pairs|\leq25|L|\leq50Z_{t-1}$, and so $A=O(Z_{t-1})$.  Also, $n'_\bfj$ is stochastically dominated
by a Poisson random variable with mean $r^2\lr=\lambda/\pi=O(1)$;  thus 
$$
\Ex(C)=\sum_{\bfi:b_\bfi>0}25\lambda/\pi=O(Z_{t-1}).
$$
Note that the $n'_\bfj$ new points in grid square $\Gamma_\bfj$ are freshly generated during the 
current iteration, and are discarded before the end of the iteration, and so have no 
effect on the subsequent evolution of the algorithm;  the sequence $Z_0,\alpha Z_1,\alpha^2 Z_2,\ldots$ 
remains a supermartingale, and the analysis in Lemma~\ref{lem:runtime} is unaffected.  (All other 
estimates will be deterministic, so this is the only point where we need to consider the
potential for conditioning.)

The remaining term $B$ may be bounded as follows:
\begin{align*}
B&=\sum_{(\bfi,\bfj)\in\pairs}b_\bfi n_\bfj\\
&\leq \sum_{(\bfi,\bfj)\in\pairs}1+2\binom{b_\bfi}2+2\binom{n_\bfj}2\\
&\leq \sum_{\bfi:b_\bfi>0}25+50\sum_{\bfi}\binom{b_\bfi}2 
+50\sum_{\bfj}\binom{n_\bfj}2\\
&=O(Z_{t-1})+O(Z_{t-1})+O(Z_t),
\end{align*}
Here we use the fact that any pair of distinct points in a single grid square is certainly a bad pair.

The final task is to compute the new set of bad points.  Since each of the new bad pairs must involve 
at least one point in~$P^S$,
the time to complete this task is proportional to
$$
\sum_{\pairs'}1+ n_\bfi p_\bfj\leq A'+B'+C'
$$
where 
$$
\pairs'=\{(\bfi,\bfj):n_\bfi>0\text{ and }\infnorm{\bfj-\bfi}\leq2\}
$$
and 
$$
A'=\sum_{\pairs'}2,\quad B'=\sum_{\pairs'}2\binom{n_\bfi}2 \quad\text{and}\quad C'=\sum_{\pairs'}2\binom{p_\bfj}2.
$$
Now any grid square $\Gamma_\bfi$ with $n_\bfi>0$ must be within distance 2 of a grid square in 
the list~$L$.  Therefore, by similar reasoning to that used earlier,
\begin{align*}
A'&=\sum_{\bfi:n_\bfi>0}25\times 2=25|L|\times50\leq 2500Z_{t-1},\\
B'&\leq \sum_{\bfi}25\times 2\binom{n_\bfi}2\leq 50 Z_t\\
\noalign{\noindent and} 
C'&\leq\sum_{\bfj} 25\times 2\binom{p_\bfj}2\leq 50 Z_t.
\end{align*}
So again, the time for this phase of the loop is bounded by 
$$A'+B'+C'=O(Z_{t-1})+O(Z_{t})+O(Z_t)=O(Z_{t-1})+O(Z_t).$$

It remains to analyse the initialisation phase.  
Generating the realisation of a Poisson point process of intensity $\lr$ in the unit square clearly takes time $O(r^{-2})$.  
To identify the bad pairs, we cycle through pairs of grid squares $(\Gamma_\bfi.\Gamma_\bfj)$ with $\infnorm{\bfj-\bfi}\leq2$;  
there are $O(r^{-2})$ of these.  
An argument identical to those above shows that the time taken to identify the bad pairs is $O(Z_0)$.  
We are almost done, but we do need a better estimate for $Z_0$ that the one used in Lemma~\ref{lem:runtime}, which was $O(r^{-4})$.  (This crude estimate was adequate at the time, as we were only interested in the logarithm of this quantity.)
But we can now see that bad pairs can only come from pairs of grid squares separated by distance at most two.  
There are $O(r^{-2})$ of these, and each of them generates $O(1)$ bad pairs in expectation, so that $Z_0=O(r^{-2})$.
We saw in Lemma~\ref{lem:runtime} that $Z_0,\alpha Z_1,\alpha^2Z_2$, is a supermartingale with $\alpha>1$ (with the convention that $Z_t=0$ for $t>T$).  
Thus the overall runtime of Algorithm~\ref{alg:PRS} is 
$$O\big(\Ex(Z_0+Z_1+Z_2+\cdots)\big)=\frac\alpha{\alpha-1}\,O(\Ex(Z_0))=O(r^{-2})$$ 
in expectation, assuming $\lambda<\lc$.
\end{proof}

Theorem \ref{thm:runtime} is obtained by combining Lemmas~\ref{lem:runtime}, \ref{lem:runtime-2} and \ref{lem:runtimeloop}.

\section{Three or more dimensions}

In higher dimensions, the hard disk model is known as the hard spheres model.
Everything in Sections \ref{sec:correctness} and~\ref{sec:runtime} carries 
across to $d>2$ dimensions with little change.  For general $d$, the appropriate 
scaling for the intensity is $\lrd=\lambda/(v_dr^d)$, where $v_d$ is the 
volume of a ball of unit radius in $d$~dimensions.  Note that in a realisation
of a Poisson point process with intensity $\lrd$, the expected number of points in
a ball of radius~$r$ is~$\lambda$. 

The analogue of equation~\eqref{eq:doubleint} is 
\[
k'\leq \int_{S_t}\lrd\int_{[0,1]^d}\lrd\,\mathbf{1}_{\|x-y\|\leq2r}\,dy\,dx,
\]
which leads to
\[
k'\leq 2^{2d+1}\lambda^2k.
\]
So setting $\lc=2^{-(d+\frac12)}$ we find that $\alpha =k/k'>1$ for any $\lambda<\lc$.
It follows that the runtime of partial rejection sampling 
is $O(\log r)$ for any $\lambda<\lc$.

By a result of Jenssen, Joos and Perkins~\cite{JJP19}, 
we lose just a constant factor when translating 
from intensity~$\lambda$ to packing density~$\alpha$.  
(It is partly to connect with their work, we measure intensity in terms of the expected number of points
in a ball of radius~$r$.)  In the proof of \cite[Thm~2]{JJP19}, the following 
inequality is derived:
\[
\alpha\geq\inf_z\max\big\{\lambda e^{-z},\,2^{-d}\exp[-2\cdot3^{d/2}\lambda]\cdot z\big\}.
\]
Assuming $\lambda\leq\lc$, which holds in the range of validity of 
our algorithm, we have $\sqrt2\lambda\leq2^{-d}$ and hence
\begin{align*}
\alpha&\geq\inf_z\max\big\{\lambda e^{-z},\,\sqrt2\lambda\exp[-\sqrt2(3/4)^{d/2}]\cdot z\big\}\\
&=c_d\lambda,\\
\noalign{\noindent where}
c_d&=\inf_z\max\big\{e^{-z},\,\sqrt2\exp[-\sqrt2(3/4)^{d/2}]\cdot z\big\}.
\end{align*}
Note that $(c_d)$ is monotonically increasing, with $c_2=0.42220+$, and $\lim_{d\to\infty}c_d=0.63724+$.
It follows that we can reach expected packing 
density $\Omega(2^{-d})$ with $O(\log r^{-1})$ expected iterations.  
This is currently the best that 
can be achieved by any provably correct sampling algorithm with polynomial (in $1/r$) 
runtime~\cite{KMM03}.  The asymptotically 
best packing density currently rigorously known is $d2^{-d}$, but achieving this 
would require $\lambda$ to grow exponentially fast in~$d$.
This is clearly beyond the capability of partial 
rejection sampling, but also beyond the capability of any known efficient sampling algorithm. 

\section*{Acknowledgements}

We thank Mark Huber and Will Perkins for inspiring conversations and bringing the hard disks model to our attention.

\bibliographystyle{plain}
\bibliography{PRS}

\begin{thebibliography}{10}

\bibitem{Coh17}
Henry Cohn.
\newblock A conceptual breakthrough in sphere packing.
\newblock {\em Notices Amer. Math. Soc.}, 64(2):102--115, 2017.

\bibitem{CKMRV17}
Henry Cohn, Abhinav Kumar, Stephen~D. Miller, Danylo Radchenko, and Maryna
  Viazovska.
\newblock The sphere packing problem in dimension 24.
\newblock {\em Ann. of Math. (2)}, 185(3):1017--1033, 2017.

\bibitem{EngelEtAl}
Michael Engel, Joshua~A. Anderson, Sharon~C. Glotzer, Masaharu Isobe,
  Etienne~P. Bernard, and Werner Krauth.
\newblock Hard-disk equation of state: First-order liquid-hexatic transition in
  two dimensions with three simulation methods.
\newblock {\em Phys. Rev. E}, 87:042134, Apr 2013.

\bibitem{HardDisksICALP}
Heng Guo and Mark Jerrum.
\newblock Perfect simulation of the hard disks model by partial rejection
  sampling.
\newblock In {\em {ICALP}}, volume 107 of {\em LIPIcs}, pages 69:1--69:10.
  Schloss Dagstuhl - Leibniz-Zentrum fuer Informatik, 2018.

\bibitem{GJL17}
Heng Guo, Mark Jerrum, and Jingcheng Liu.
\newblock Uniform sampling through the {L}ov\'{a}sz local lemma.
\newblock {\em J. ACM}, 66(3):Art. 18, 31, 2019.

\bibitem{Hal05}
Thomas~C. Hales.
\newblock A proof of the {K}epler conjecture.
\newblock {\em Ann. of Math. (2)}, 162(3):1065--1185, 2005.

\bibitem{HM14}
Thomas~P. Hayes and Cristopher Moore.
\newblock Lower bounds on the critical density in the hard disk model via
  optimized metrics.
\newblock {\em CoRR}, abs/1407.1930, 2014.

\bibitem{JJP19}
Matthew Jenssen, Felix Joos, and Will Perkins.
\newblock On the hard sphere model and sphere packings in high dimensions.
\newblock {\em Forum Math. Sigma}, 7:e1, 19, 2019.

\bibitem{KMM03}
Ravi Kannan, Michael~W. Mahoney, and Ravi Montenegro.
\newblock Rapid mixing of several {M}arkov chains for a hard-core model.
\newblock In {\em ISAAC}, pages 663--675, 2003.

\bibitem{Kendall98}
Wilfrid~S. Kendall.
\newblock Perfect simulation for the area-interaction point process.
\newblock In {\em Probability towards 2000 ({N}ew {Y}ork, 1995)}, volume 128 of
  {\em Lect. Notes Stat.}, pages 218--234. Springer, New York, 1998.

\bibitem{KM00}
Wilfrid~S. Kendall and Jesper M{\o}ller.
\newblock Perfect simulation using dominating processes on ordered spaces, with
  application to locally stable point processes.
\newblock {\em Adv. Appl. Probab.}, 32(3):844–--865, 2000.

\bibitem{Lindvall}
Torgny Lindvall.
\newblock On {S}trassen's theorem on stochastic domination.
\newblock {\em Electron. Comm. Probab.}, 4:51--59, 1999.

\bibitem{Low00}
Hartmut L{\"o}wen.
\newblock Fun with hard spheres.
\newblock In Klaus~R. Mecke and Dietrich Stoyan, editors, {\em Statistical
  Physics and Spatial Statistics}, pages 295--331, 2000.

\bibitem{MRRTT53}
Nicholas Metropolis, Arianna~W. Rosenbluth, Marshall~N. Rosenbluth, Augusta~H.
  Teller, and Edward Teller.
\newblock Equation of state calculations by fast computing machines.
\newblock {\em J. Chem. Phys.}, 21(6):1087--1092, 1953.

\bibitem{MJM17}
S.~B. Moka, S.~Juneja, and M.~R.~H. Mandjes.
\newblock Perfect sampling for {G}ibbs processes with a focus on hard-sphere
  models.
\newblock {\em ArXiv}, abs/1705.00142, 2017.

\bibitem{Per16}
Will Perkins.
\newblock Birthday inequalities, repulsion, and hard spheres.
\newblock {\em Proc. Amer. Math. Soc.}, 144(6):2635--2649, 2016.

\bibitem{Preston}
Chris Preston.
\newblock Spatial birth-and-death processes (with discussion).
\newblock {\em Bull. Inst. Internat. Statist.}, 46(2):371--391, 405--408
  (1975), 1975.

\bibitem{Strassen}
V.~Strassen.
\newblock The existence of probability measures with given marginals.
\newblock {\em Ann. Math. Statist.}, 36:423--439, 1965.

\bibitem{Str75}
David~J. Strauss.
\newblock A model for clustering.
\newblock {\em Biometrika}, 62(2):467--475, 1975.

\bibitem{Via17}
Maryna~S. Viazovska.
\newblock The sphere packing problem in dimension 8.
\newblock {\em Ann. of Math. (2)}, 185(3):991--1015, 2017.

\bibitem{Wel18}
Jake Wellens.
\newblock A note on partial rejection sampling for the hard disks model in the
  plane.
\newblock {\em ArXiv}, abs/1808.03367, 2018.

\end{thebibliography}

\end{document}